\documentclass[12pt]{amsart}

\usepackage{amssymb,latexsym}
\usepackage{bm}
\usepackage{mathrsfs}
\usepackage{graphicx}
\usepackage{psfrag}
\usepackage{wrapfig}
\usepackage{color}
\usepackage{enumitem}
\setlist[enumerate]{parsep=0pt plus 4pt,topsep=0pt plus 4pt}
\usepackage{url}
\usepackage{hyperref}
\usepackage{tikz}
\definecolor{darkblue}{RGB}{0,0,160}
\hypersetup{colorlinks,citecolor=black,filecolor=black,%
            linkcolor=darkblue,urlcolor=darkblue}
\allowdisplaybreaks

\newcommand{\excise}[1]{}

\newtheorem{thm}{Theorem}[section]
\newtheorem{lemma}[thm]{Lemma}

\newtheorem{cor}[thm]{Corollary}
\newtheorem{prop}[thm]{Proposition}

\newtheorem{hyp}[thm]{Hypothesis}

\theoremstyle{definition}

\newtheorem{remark}[thm]{Remark}
\newtheorem{defn}[thm]{Definition}

\numberwithin{equation}{section}

\newcounter{separated}

\newcommand\DD{\mathscr{D}}
\newcommand\EE{\mathscr{E}}
\newcommand\FF{\mathscr{F}}
\newcommand\GG{\mathscr{G}}

\newcommand\OO{\mathcal{O}}

\newcommand\RR{{\mathbb R}}
\newcommand\UU{\mathscr{U}}
\newcommand\ZZ{{\mathbb Z}}

\newcommand\ii{{\mathbf i}}
\newcommand\jj{{\mathbf j}}
\newcommand\kk{\Bbbk}

\newcommand\cN{N}

\newcommand\oD{\hspace{.3ex}\ol{\hspace{-.3ex}D\hspace{-.05ex}}\hspace{.05ex}}

\renewcommand\phi{\varphi}

\newcommand\ale{{\mathord{\mathrm{ale}}}}

\newcommand\con{{\mathord{\mathrm{con}}}}

\newcommand\ord{{\mathord{\mathrm{ord}}}}

\newcommand\from{\leftarrow}

\newcommand\onto{\twoheadrightarrow}

\newcommand\spot{{\hbox{\raisebox{1pt}{\tiny$\scriptscriptstyle\bullet$}}}}

\newcommand\simto{\mathrel{\!\ooalign{$\fillrightmap$\cr\raisebox{.75ex}{$\,\sim\ \hspace{.2ex}$}}}}

\newcommand\fillrightmap{\mathord- \mkern-6mu
	\cleaders\hbox{$\mkern-2mu \mathord- \mkern-2mu$}\hfill
	\mkern-6mu \mathord\rightarrow}

\renewcommand\iff{\Leftrightarrow}
\renewcommand\epsilon{\varepsilon}
\renewcommand\implies{\Rightarrow}

\newcommand\dx[1][]{\delta^{\hspace{.1ex}\xi}}

\newcommand\ol[1]{{\overline{#1}}}

\definecolor{lightred}{rgb}{1,.3,.3}

\begin{document}

\title[Stratifications of vector spaces from conical microsupport]%
      {Stratifications of real vector spaces from\\%
	constructible sheaves with conical microsupport}
\author{Ezra Miller}
\address{Mathematics Department\\Duke University\\Durham, NC 27708}
\urladdr{\url{http://math.duke.edu/people/ezra-miller}}

\makeatletter
  \@namedef{subjclassname@2020}{\textup{2020} Mathematics Subject Classification}
\makeatother

\date{22 January 2023}

\begin{abstract}
Interpreting the syzygy theorem for tame modules over posets in the
setting of derived categories of subanalytically constructible sheaves
proves two conjectures due to Kashiwara and Schapira concerning the
existence of stratifications of real vector spaces that play well with
sheaves having microsupport in a given cone or, equivalently, sheaves
in the corresponding conic topology.
\end{abstract}
\maketitle

\setcounter{tocdepth}{2}
\tableofcontents

\section{Introduction}\label{s:intro}

\setcounter{tocdepth}{-1}
\subsection*{Overview and motivation}\label{sub:overview}
\setcounter{tocdepth}{2}

Persistent homology with multiple real parameters can be phrased in
more or less equivalent ways using multigraded modules
(e.g.~\cite{multiparamPH,knudson2008,hom-alg-poset-mods}), or sheaves
(e.g.~\cite{curry-thesis,curry2019}), or functors
(e.g.~\cite{scolamiero-chacholski-lundman-ramanujam-oberg16}), or
derived categories (e.g.~\cite{kashiwara-schapira2018}).  All of these
descriptions have in common an underlying partially ordered set
indexing a family of vector spaces, and this is interpreted under
increasing layers of abstraction.

The simplest objects at any level of abstraction are the \emph{indicator}
objects, which place a single copy of the ground field~$\kk$
at every point of an interval in the underlying poset~$Q$, meaning an
intersection of an upset of~$Q$ with a downset of~$Q$.  (The
terminology is most clear when $Q = \RR$, where ``interval'' has its
usual meaning.)  Among the indicator objects are those supported on
the upsets and downsets themselves; over $Q = \RR$ these objects are
free and injective, respectively.  Furthermore, when $Q$ is totally
ordered, every object is a direct sum of indicator objects
\cite{crawley-boevey2015}.  The theory for more general posets,
including partially ordered real vector spaces, has in large part
revolved around relating general objects as closely as possible to
indicator objects, particularly where algorithmic computation is
concerned.

Indeed, the foundations for the ideas in this paper, both from
Kashiwara--Schapira \cite{kashiwara-schapira2018,
kashiwara-schapira2019} and the author \cite{hom-alg-poset-mods,
prim-decomp-pogroup, essential-real} (see also \cite{qr-codes}), lies
in algorithmic computation with persistent homology.  To that end,
effective methods demand concrete representatives of derived sheaves
and stratifications of their support.  Kashiwara and Schapira, in
\cite[Conjecture~3.20]{kashiwara-schapira2019} (which they had
previously stated as \cite[Conjecture~4.19]{kashiwara-schapira2018})
and \cite[Conjecture\,3.17]{kashiwara-schapira2017}, assert that
derived sheaves in principle possess such concrete representatives.
Corollaries~\ref{c:PL} and~\ref{c:strat} achieve more than mere
existence: the engine behind their proofs, Theorem~\ref{t:res},
produces concrete structures via the syzygy theorem for poset modules
(Theorem~\ref{t:syzygy-complexes}), which is specifically designed to
extract resolutions algorithmically~\cite{hom-alg-poset-mods}.

Situations abound where concrete resolutions could be essential for
algorithmic persistent homology.  For example, bifiltration of a
semialgebraic space by two semialgebraic functions yields bipersistent
homology that is an $\RR^2$-module which should be tame and hence have
finite indicator resolutions by upsets or downsets.  (This statement
requires proof and might be subtle or even false in the subanalytic
instead of semialgebraic context.)  This scenario is fundamental to
motivating applications such as summarizing shape in biology
\cite{fruitFlyModuli} or probability distributions in statistics
\cite{rolle-scoccola2020}.  Tameness in this biparameter setting
connects to recent Morse-theoretic stratification perspectives by
Budney and Kaczynski \cite{budney-kaczynski2021} as well as by Assif
and~Baryshnikov~\cite{assif-baryshnikov2021}.

\setcounter{tocdepth}{-1}
\subsection*{Conjectures and proofs}\label{sub:conjectures}
\setcounter{tocdepth}{2}

The conjectures of Kashiwara and Schapira are phrased in the most
abstract derived setting.  They posit, roughly speaking, that every
object can be directly related to indicator objects, either by
stratification of its support or---more strongly---by resolution.
More precisely, the first conjecture concerns the relation between, on
one hand, constructibility of sheaves on real vector spaces in the
derived category with microsupport restricted to a cone, and on the
other hand, stratification of the vector space in a manner compatible
with the cone \cite[Conjecture~3.17]{kashiwara-schapira2017}
	\footnote{Bibliographic note: this conjecture appears in~v3
	(the version cited here) and earlier versions of the cited
	arXiv preprint.  It does not appear in the published version
	\cite{kashiwara-schapira2018}, which is~v6 on the arXiv.  The
	published version is cited where it is possible to do so, and
	v3 \cite{kashiwara-schapira2017} is cited otherwise.}
(Corollary~\ref{c:strat}).  The second concerns piecewise linear (PL)
objects in this context, particularly existence of polyhedrally
structured resolutions that, in principle, lend themselves to explicit
or algorithmic computation
\cite[Conjecture~4.19]{kashiwara-schapira2018} =
\cite[Conjecture~3.20]{kashiwara-schapira2019}~(\mbox{Corollary}~\ref{c:PL}).

This note uses the most elementary poset module setting
\cite{hom-alg-poset-mods} to prove these conjectures.  Both follow
immediately from Theorem~\ref{t:res} here, which translates the
relevant real-vector-space special cases of the syzygy theorem for
complexes of poset modules \cite[Theorem~6.17]{hom-alg-poset-mods}
(reviewed in Section~\ref{s:syzygy}) into the language of derived
categories of constructible sheaves with conic microsupport or under a
conic~topology (reviewed in Section~\ref{s:cones}).

The syzygy theorem \cite[Theorems~6.12 and~6.17]{hom-alg-poset-mods}
leverages relatively weak topological framework into powerful
homological structure: over any poset~$Q$ it enhances a constant
subdivision---a partition of~$Q$ into finitely many regions over which
the given module or complex is constant---to a more controlled
subdivision (a~finite encoding \cite[\S4]{hom-alg-poset-mods}), and
even to a finite resolution by upset modules and a finite resolution
by downset modules, whose pieces play well with the ambient
combinatorics.  These resolutions are analogues over arbitrary posets
of free and injective resolutions for modules over the poset~$\ZZ^n$
\cite{goto-watanabe1978} (see \cite{injAlg} or \cite[Chapter~11]{cca}
for background, or \cite[\S5]{hom-alg-poset-mods} for a treatment in
the present context) or over the poset~$\RR$.  Crucially, any
available supplementary geometry---be it subanalytic, semialgebraic,
or piecewise-linear, for instance---is preserved.

In the context of a partially ordered real vector space~$Q$ with
positive cone~$Q_+$, the enhancement afforded by the syzygy theorem
produces a $Q_+$-stratification from an arbitrary subanalytic
triangulation.  If the triangulation is subordinate to a given
constructible derived $Q_+$-sheaf, meaning an object in the bounded
derived category of constructible sheaves with microsupport contained
in the negative polar cone of~$Q_+$, then this enhancement produces
$Q_+$-structured resolutions of the given sheaf.  This makes the two
conjectures into special cases of the syzygy theorem.

While sheaves with conical microsupport (see
Section~\ref{sub:microsupport}) are equivalent to the more elementary
sheaves in the conic topology (see Section~\ref{sub:topology}), as has
been known from the outset \cite{kashiwara-schapira1990}
(see Theorem~\ref{t:microsupport}), the notion of constructibility has
until now been available only on the microsupport side.  The results
here assert that constructibility can be detected entirely with the
more rigid conic topology, via tameness, without appealing to
subanalytic triangulations in the more flexible analytic topology, a
point emphasized by Theorem~\ref{t:res'}.  More broadly, for
applications in persistent homology the input is usually a sheaf in
the conic topology induced by a given cone~$Q_+$ instead of a sheaf in
the ordinary topology with microsupport in the negative polar
cone~$Q_+^\vee$, so the main results, namely Theorem~\ref{t:res},
Corollary~\ref{c:PL}, and Corollary~\ref{c:strat}, are restated using
conic sheaves in Theorem~\ref{t:res'}, Corollary~\ref{c:PL'}, and
Corollary~\ref{c:strat'}.

\setcounter{tocdepth}{-1}
\subsection*{Poset modules vs.~constructible sheaves}\label{sub:comparison}
\setcounter{tocdepth}{2}

The theory in \cite{hom-alg-poset-mods, prim-decomp-pogroup,
essential-real} was developed simultaneously and independently from
\cite{kashiwara-schapira2018, kashiwara-schapira2019}
(cf.~\cite{qr-codes}).  Having made the connection between these
approaches, it is worth comparing them~in~\mbox{detail}.$\!$

The syzygy theorem \cite[Theorems~6.12 and~6.17]{hom-alg-poset-mods}
and its combinatorial underpinnings involving poset encoding
\cite[\S4]{hom-alg-poset-mods} hold over arbitrary posets; see
Section~\ref{s:syzygy} here for indications toward this generality.
When the poset is a real vector space, the constructibility
encapsulated by topological tameness (Definition~\ref{d:tame}) has no
subanalytic, algebraic, or piecewise-linear hypothesis, although these
additional structures are preserved by the syzygy theorem transitions.
For example, the upper boundary of a downset in the plane with the
usual componentwise partial order could be the graph of any continuous
weakly decreasing function, among other things, and could be present
(i.e.,~the downset is closed) or absent (i.e.,~the downset is open),
or somewhere in between (e.g., a Cantor set could be missing).  The
conic topology in \cite{kashiwara-schapira2018} or
\cite{kashiwara-schapira2019} specializes at the outset to the case of
a partially ordered real vector space, and it allows only subanalytic
or polyhedral regions, respectively, with upsets having closed lower
boundaries and downsets having open upper boundaries.  The
constructibility in
\cite{kashiwara-schapira2018,kashiwara-schapira2019} is otherwise
essentially the same as tameness here (Theorem~\ref{t:res'}), except
that tameness requires constant subdivisions to be finite, whereas
constructibility in the derived category allows constant subdivisions
to be locally finite.  That said, this agreement of constructibility
with locally finite tameness that is subanalytic or PL, more or less
up to boundary considerations, is visible in
\cite{kashiwara-schapira2017} or \cite{kashiwara-schapira2019} only
via conjectures, namely the ones proved here in Section~\ref{s:strat}
using the general poset methods.

The theory of primary decomposition in \cite{prim-decomp-pogroup}
requires the poset to be a partially ordered group whose positive cone
has finitely many faces.  These can be integer or real or something in
between, but the finiteness is essential for primary decomposition in
any of these settings; see \cite[Example~5.9]{prim-decomp-pogroup}.
Local finiteness allowed by constructibility in
\cite{kashiwara-schapira2018} does not provide a remedy, although it
is possible that the PL hypothesis in \cite{kashiwara-schapira2019}
does.  In either the integer or real case, detailed understanding of
the topology results in a stronger theory of primary decomposition
than over an aribtrary polyhedral group, with much more complete
supporting commutative algebra~\cite{essential-real}.

Most of the remaining differences between the developments in
\cite{hom-alg-poset-mods,prim-decomp-pogroup,essential-real} and
those in \cite{kashiwara-schapira2018,kashiwara-schapira2019}, beyond
the types of allowed functions and the shapes of allowed regions, is
the behavior allowed on boundaries of regions.  That difference is
accounted for by the transition between the conic topology and the
Alexandrov topology, the distinction being that the Alexandrov
topology has for its open sets all upsets, whereas the conic topology
has only the upsets that are open in the usual topology.  This
distinction is explored in detail by Berkouk and Petit
\cite{berkouk-petit2019}.  It is intriguing that ephemeral modules are
undetectable metrically \cite[Theorem~4.22]{berkouk-petit2019} but
their presence here brings indispensable insight into homological
behavior in the~conic~\mbox{topology}.

\setcounter{tocdepth}{-1}
\subsection*{Acknowledgements}\label{sub:acknowledgements}
\setcounter{tocdepth}{2}

Pierre Schapira gave helpful comments on a draft of this paper, as did
a referee.  Portions of this work were funded by NSF
grant~DMS-1702395.

\section{Syzygy theorem for poset modules}\label{s:syzygy}

This section recalls concepts surrounding modules over posets,
concluding with a statement (Theorem~\ref{t:syzygy-complexes}) of the
relevant special case of the syzygy theorem for complexes of poset
modules \cite[Theorem~6.17]{hom-alg-poset-mods}.  For reference, the
definitions here correspond to
\cite[Definitions~2.1,
2.6,
2.11,
2.14,
2.15,
4.27,
3.1,
3.14,
6.1,
and 6.16]{hom-alg-poset-mods},
%
sometimes special cases thereof.

\subsection{Tame poset modules}\label{sub:tame}

\begin{defn}\label{d:poset-module}
Let $Q$ be a partially ordered set (\emph{poset}) and~$\preceq$ its
partial order.  A \emph{module over~$Q$} (or a \emph{$Q$-module})
is
\begin{itemize}
\item%
a $Q$-graded vector space $M = \bigoplus_{q\in Q} M_q$ with
\item%
a homomorphism $M_q \to M_{q'}$ whenever $q \preceq q'$ in~$Q$
such that
\item%
$M_q \to M_{q''}$ equals the composite $M_q \to M_{q'} \to
M_{q''}$ whenever $q \preceq q' \preceq q''$.
\end{itemize}
A \emph{homomorphism} $M \to \cN$ of $Q$-modules is a
degree-preserving linear map, or equivalently a collection of vector
space homomorphisms $M_q \to \cN_q$, that commute with the structure
homomorphisms $M_q \to M_{q'}$ and $\cN_q \to \cN_{q'}$.
\end{defn}

\begin{defn}\label{d:constant-subdivision}
Fix a $Q$-module~$M$.  A \emph{constant subdivision} of~$Q$
\emph{subordinate} to~$M$ is a partition of~$Q$ into
\emph{constant regions} such that for each constant region~$I$ there
is a single vector space~$M_I$ with an isomorphism $M_I \to M_\ii$ for
all $\ii \in I$ that \emph{has no monodromy}: if $J$ is some (perhaps
different) constant region, then all comparable pairs $\ii \preceq
\jj$ with $\ii \in I$ and $\jj \in J$ induce the same composite
homomorphism $M_I \to M_\ii \to M_\jj \to M_J$.
\end{defn}

\begin{defn}\label{d:tame}
Fix a poset~$Q$ and a $Q$-module~$M$.
\mbox{}
\begin{enumerate}
\item\label{i:finite-constant-subdiv}%
A constant subdivision of~$Q$ is \emph{finite} if it has finitely
many constant regions.

\item\label{i:Q-finite}%
The $Q$-module~$M$ is \emph{$Q$-finite} if its components $M_q$
have finite dimension over~$\kk$.

\item\label{i:tame}%
The $Q$-module~$M$ is \emph{tame} if it is $Q$-finite and $Q$
admits a finite constant subdivision subordinate to~$M$.
\end{enumerate}
\end{defn}

\subsection{Real partially ordered groups}\label{sub:real}

\begin{defn}\label{d:pogroup}
An abelian group~$Q$ is \emph{partially ordered} if it is generated
by a submonoid~$Q_+$, called the \emph{positive cone}, that has
trivial unit group.  The partial order is: $q \preceq q' \iff q' - q
\in Q_+$.  A partially ordered group is
\begin{enumerate}
\item\label{i:real}%
\emph{real} if the underlying abelian group is a real vector space of
finite dimension;
\item\label{i:subanalytic'}%
\emph{subanalytic} if, in addition, $Q_+$ is subanalytic,
(see \cite[\S8.2]{kashiwara-schapira1990} for the definition);

\item\label{i:polyhedral}%
\emph{polyhedral} if, in addition, $Q_+$ is a \emph{convex
polyhedron}: an intersection of finitely many half-spaces, each either
closed or open.
\end{enumerate}
\end{defn}

\begin{defn}\label{d:auxiliary-hypotheses}
A partition of a real partially ordered group~$Q$ into subsets is
\begin{enumerate}
\item\label{i:subanalytic}%
\emph{subanalytic} if the subsets are subanalytic sets, and

\item\label{i:PL}%
\emph{piecewise linear (PL)} if the subsets are finite unions of
convex polyhedra.
\end{enumerate}
A module over a subanalytic or polyhedral real partially ordered
group~$Q$ is \emph{subanalytic} or \emph{PL}, respectively, if the
module is tamed by a subordinate finite constant subdivision of the
corresponding type.
\end{defn}

\subsection{Complexes and resolutions of poset modules}\label{sub:complexes}

\begin{defn}\label{d:tame-morphism}
A homomorphism $\phi: M \to N$ of $Q$-modules is \emph{tame} if $Q$
admits a finite constant subdivision subordinate to both $M$ and~$N$
such that for each constant region~$I$ the composite homomorphism $M_I
\to M_\ii \to N_\ii \to N_I$ does not depend~on~$\ii \in I$.  The map
$\phi$ is subanalytic or PL if this constant subdivision~is.
\end{defn}

\begin{defn}\label{d:indicator}
The vector space $\kk[Q] = \bigoplus_{q\in Q} \kk$ that assigns $\kk$
to every point of the poset~$Q$ is a $Q$-module with identity maps
on~$\kk$.  More generally,
\begin{enumerate}
\item\label{i:upset}%
an \emph{upset} (also called a \emph{dual order ideal}) $U \subseteq
Q$, meaning a subset closed \mbox{under} going upward in~$Q$ (so $U +
Q_+ = U$, when $Q$ is a partially ordered group) determines an
\emph{indicator submodule} or \emph{upset module} $\kk[U] \subseteq
\kk[Q]$; and
\item\label{i:downset}%
dually, a \emph{downset} (also called an \emph{order ideal}) $D
\subseteq Q$, meaning a subset closed under going downward in~$Q$ (so
$D - Q_+ = D$, when $Q$ is a partially ordered group) determines an
\emph{indicator quotient module} or \emph{downset module} $\kk[Q]
\onto \kk[D]$.
\end{enumerate}
When $Q$ is a subposet of a partially ordered real vector space, an
indicator module of either sort is subanalytic or PL if the
corresponding upset or downset is of the~same~type.
\end{defn}

\begin{defn}\label{d:connected-homomorphism}
Let each of $S$ and $S'$ be a nonempty intersection of an upset in a
poset~$Q$ with a downset in~$Q$, so $\kk[S]$ and $\kk[S']$ are
subquotients of~$\kk[Q]$.  A homomorphism $\phi: \kk[S] \to \kk[S']$
is \emph{connected} if there is a scalar $\lambda \in \kk$ such that
$\phi$ acts as multiplication by~$\lambda$ on the copy of~$\kk$ in
degree $q$ for all $q \in S \cap S'$.
\end{defn}

\begin{defn}\label{d:resolutions}
Fix any poset~$Q$ and a $Q$-module~$M$.
\begin{enumerate}
\item%
An \emph{upset resolution} of~$M$ is a complex~$F_\spot$ of
$Q$-modules, each a direct sum of upset submodules of~$\kk[Q]$, whose
differential $F_i \to F_{i-1}$ decreases homological degrees, has
components $\kk[U] \to \kk[U']$ that are connected, and has only one
nonzero homology $H_0(F_\spot) \cong M$.

\item%
A \emph{downset resolution} of~$M$ is a complex~$E^\spot$ of
$Q$-modules, each a direct sum of downset quotient modules
of~$\kk[Q]$, whose differential $E^i \to E^{i+1}$ increases
cohomological degrees, has components $\kk[D'] \to \kk[D]$ that are
connected, and has only one nonzero cohomology $H^0(E^\spot)
\cong\nolinebreak M$.
\end{enumerate}\setcounter{separated}{\value{enumi}}
An upset or downset resolution is called an \emph{indicator
resolution} if the up- or down- nature is unspecified.  The
\emph{length} of an indicator resolution is the largest
(co)homological degree in which the complex is nonzero.  An indicator
resolution
\begin{enumerate}\setcounter{enumi}{\value{separated}}
\item%
is \emph{finite} if the number of indicator module summands is finite,

\item%
\emph{dominates} a constant subdivision of~$M$ if the subdivision or
encoding is subordinate to each indicator summand, and

\item\label{i:auxiliary-resolution}%
is \emph{subanalytic} or \emph{PL} if $Q$ is a subposet of a real
partially ordered group and the resolution dominates a constant
subdivision of the corresponding type.
\end{enumerate}
\end{defn}

\begin{defn}\label{d:tame-complex}
Fix a complex $M^\spot$ of modules over a poset~$Q$.
\begin{enumerate}
\item%
$M^\spot$ is \emph{tame} if its modules and morphisms are tame
(Definitions~\ref{d:tame} and~\ref{d:tame-morphism}).

\item%
A constant subdivision is \emph{subordinate} to~$M^\spot$ if it is
subordinate to all of the modules and morphisms therein, and then
$M^\spot$ is said to \emph{dominate} the subdivision.

\item%
An \emph{upset resolution} of~$M^\spot$ is a complex of $Q$-modules in
which each $F^i$ is a direct sum of upset modules and the components
$\kk[U] \to \kk[U']$ are connected, with a homomorphism $F^\spot \to
M^\spot$ of complexes inducing an isomorphism on homology.

\item%
A \emph{downset resolution} of~$M^\spot$ is a complex of
$Q$-modules in which each $E^i$ is a direct sum of downset modules
and the components $\kk[D] \to \kk[D']$ are connected, with a
homomorphism $M^\spot \to E^\spot$ of complexes inducing an
isomorphism on homology.
\end{enumerate}
These resolutions are \emph{finite}, or \emph{dominate} a constant
subdivision, or are \emph{subanalytic} or \emph{PL} as in
Definition~\ref{d:resolutions}.
\end{defn}

\subsection{Syzygy theorem for complexes of poset modules}\label{sub:syzygy}

\mbox{}
\medskip

\noindent
Only certain aspects of the full syzygy theorem
\cite[Theorem~6.17]{hom-alg-poset-mods} are required, so those are
isolated here.

\begin{thm}\label{t:syzygy-complexes}
A bounded complex $M^\spot$ of modules over a poset~$Q$ is tame if and
only if it admits one, and hence all, of the following:
\begin{enumerate}
\item\label{i:syzygy-tame}%
a finite constant subdivision of~$Q$ subordinate to~$M^\spot$; or





\item\label{i:upset-res}%
a finite upset resolution; or

\item\label{i:downset-res}%
a finite downset resolution; or



\item\label{i:subordinate-constant}%
a finite constant subdivision subordinate to any given one of
items~\ref{i:upset-res}--\ref{i:downset-res}.
\end{enumerate}
The statement remains true over any subposet of a real partially
ordered group if ``tame'' and all occurrences of ``finite'' are
replaced by
``PL''.
Moreover, any tame or
PL
morphism $M^\spot \to N^\spot$ lifts to a similarly well behaved
morphism of resolutions as in parts~\ref{i:upset-res}
and~\ref{i:downset-res}.  All of these results
hold in the subanalytic case if~$M^\spot$ has compact support.
\end{thm}

\section{Stratifications, topologies, and cones}\label{s:cones}

This section collects the relevant definitions and theorems regarding
constructible sheaves from the literature.  The sizeable edifice on
which the subject is built makes it unavoidable that readers seeing
some of these topics for the first time will need to consult the cited
sources for additional background.  The goal here is to bring readers
as quickly as possible to a general statement (Theorem~\ref{t:res})
while circumscribing the ingredients necessary for its proof in such a
way that those familiar with the conjectures of Kashiwara and
Schapira, specifically \cite[Conjecture~3.17]{kashiwara-schapira2017}
and \cite[Conjecture~3.20]{kashiwara-schapira2019}, can skip
seamlessly to Section~\ref{s:res} after skimming Section~\ref{s:cones}
for terminology.

To avoid endlessly repeating hypotheses, and so readers can quickly
identify when the same hypotheses are in effect, the blanket
assumption henceforth is for $Q$ to satisfy the following, where the
positive cone~$Q_+$ is \emph{full} if it has nonempty interior.

\pagebreak[2]

\begin{hyp}\label{h:full-closed-cone}
$Q$ is a real partially ordered group with closed, full,
subanalytic~$Q_+$.
\end{hyp}

\begin{remark}\label{r:used-without-comment}
Some basic notions are used freely without further comment.
\begin{enumerate}
\item%
The notion of simplicial complex here is the one in
\cite[Definition~8.1.1]{kashiwara-schapira1990}: a collection~$\Delta$
of subsets (called \emph{simplices}) of a fixed vertex set that is
closed under taking subsets (called \emph{faces}), contains every
vertex, and is locally finite in the sense that every vertex
of~$\Delta$ lies in finitely many simplices of~$\Delta$.  Any
simplicial complex~$\Delta$ has a realization~$|\Delta|$ as a
topological space, with each relatively open simplex $|\sigma|$ being
an open convex set in an appropriate affine space.

\item%
The notion of subanalytic set in an analytic manifold is as in
\cite[\S8.2]{kashiwara-schapira1990}.

\item%
The term \emph{sheaf} on a topological space here means a sheaf of
$\kk$-vector spaces.  Sometimes in the literature this word is used to
mean an object in the bounded derived category of sheaves of
$\kk$-vector spaces; for clarity here, the term \emph{derived sheaf}
is always used when an object in the derived category is intended.
\end{enumerate}
\end{remark}

\subsection{Subanalytic triangulation}\label{sub:triangulation}

\begin{defn}\label{d:subanalytic-triangulation}
Fix a real analytic manifold~$X$.
\begin{enumerate}
\item%
A \emph{subanalytic triangulation} of a subanalytic set~$Y \subseteq
X$ is a homeomorphism $|\Delta| \simto Y$ such that the image in~$Y$
of the realization~$|\sigma|$ of the relative interior of each simplex
$\sigma \in \Delta$ is a subanalytic submanifold of~$X$.

\item%
A subanalytic triangulation of~$Y$ is \emph{subordinate} to a
(derived) sheaf $\FF$ on~$X$ if $Y$ contains the support of~$\FF$ and
(every homology sheaf~of)~$\FF$ restricts to a constant sheaf on the
image in~$Y$ of every cell~$|\sigma|$.
\end{enumerate}
\end{defn}

\subsection{Subanalytic constructibility}\label{sub:constructible}

\begin{defn}\label{d:constructible-sheaf}
A (derived) sheaf on a real analytic manifold is \emph{subanalytically
weakly constructible} if there is a subanalytic triangulation
subordinate to it.  $\!$The word\hspace{-.1ex} ``weakly'' is omitted
if, in addition, the stalks have finite dimension as $\kk$-vector
spaces.
\end{defn}

\begin{remark}\label{r:constructible-sheaf}
Readers less familiar with constructibility can safely take
Definition~\ref{d:constructible-sheaf} at face value.  For readers
familiar with constructibility by other definitions,
this characterization is a nontrivial theorem, which rests on the
triangulability of subanalytic sets
\cite[Proposition~8.2.5]{kashiwara-schapira1990} and other results
concerning subanalytic stratification; see
\cite[\S8.4]{kashiwara-schapira1990} for the full proof of
equivalence, especially Theorem~8.4.2, Definition~8.4.3, and part~(a)
of the proof of Theorem~8.4.5(i) there.  Note that the modifier
``subanalytically'' in Definition~\ref{d:constructible-sheaf} does not
appear in \cite{kashiwara-schapira1990}, because the context there is
subanalytic throughout.  Also note that it makes no difference whether
one takes constructible objects in the derived category or the derived
category of constructible objects, since they yield the same result
\cite[Theorem~8.4.5]{kashiwara-schapira1990}: every constructible
derived sheaf is represented by a complex of constructible sheaves.
\end{remark}

The reason to use subanalytic triangulation instead of arbitrary
subanalytic stratification is the following, which is a step on the
way to a constant subdivision.

\begin{lemma}[{\cite[Proposition~8.1.4]{kashiwara-schapira1990}}]\label{l:tri-constant}
For a simplex~$\sigma$ in a subanalytic triangulation subordinate to a
constructible sheaf~$\FF$, there is a natural isomorphism
$\Gamma(|\sigma|,\FF) \simto \FF_x$ from the sections over~$|\sigma|$
to the stalk at every point~$x \in \sigma$.
\end{lemma}

The reason for specifically including the piecewise linear (PL)
condition in Section~\ref{s:syzygy} is for its application here, as
one of the conjectures is in that setting.  For this purpose, the
sheaf version of this particularly strong type of constructibility is
needed.

\begin{defn}\label{d:PL-constructible}
Fix $Q$ satisfying Hypothesis~\ref{h:full-closed-cone}.
\begin{enumerate}
\item\label{i:stratification}%
A~subanalytic subdivision
(Definition~\ref{d:auxiliary-hypotheses}.\ref{i:subanalytic}) of~$Q$
is \emph{subordinate} to a (derived) sheaf~$\FF$ on~$Q$ if the
restriction of~$\FF$ to every \emph{stratum} (meaning subset in the
subdivision) is constant of finite rank.

\item\label{i:PL-constructible}%
If the subanalytic subdivision is PL
(Definition~\ref{d:auxiliary-hypotheses}.\ref{i:PL}) and $Q$ is
polyhedral (Definition~\ref{d:pogroup}.\ref{i:polyhedral}), then $\FF$
is said to be \emph{piecewise linear}, abbreviated \emph{PL}.
\end{enumerate}
\end{defn}

\begin{remark}\label{r:PL}
Definition~\ref{d:PL-constructible}.\ref{i:PL-constructible} is not
verbatim the same as \cite[Definition~2.3]{kashiwara-schapira2019},
which only requires $Q$ to be a (nondisjoint) union of finitely
polyhedra on which $\FF$ is constant.  However, the notion of PL
(derived) sheaf thus defined is the same, since any finite union of
polyhedra can be refined to a finite union that is disjoint---that is,
a partition.  This refinement can be done, for example, by
expressing~$Q$ as the union of (relatively open) faces in the
arrangement of all hyperplanes bounding halfspaces defining the given
polyhedra, of which there are only finitely many.
\end{remark}

\subsection{Conic and Alexandrov topologies}\label{sub:topology}

\begin{defn}\label{d:conic-topology}
Fix a real partially ordered group~$Q$ with closed positive
cone~$Q_+$.
\begin{enumerate}
\item%
The \emph{conic topology} on~$Q$ induced by~$Q_+$ (or induced by the
partial order) consists of the upsets that are open in the ordinary
topology on~$Q$.

\item%
The \emph{Alexandrov topology} on~$Q$ induced by~$Q_+$ (or induced by
the partial order) consists of all the upsets in~$Q$.
\end{enumerate}
To avoid confusion when it might occur, write
\begin{enumerate}
\item%
$Q^\con$ for the set~$Q$ with the conic topology induced by~$Q_+$,

\item%
$Q^\ale$\, for the set~$Q$ with the Alexandrov topology induced
by~$Q_+$,~and

\item%
$Q^\ord$ for the set~$Q$ with its ordinary topology.
\end{enumerate}
\end{defn}

\begin{remark}\label{r:conic-topology}
The conic topology in Definition~\ref{d:conic-topology} is also known
as the \emph{$\gamma$-topology}, where $\gamma = Q_+$
\cite{kashiwara-schapira1990, kashiwara-schapira2018,
kashiwara-schapira2019}.  The Alexandrov topology makes just as much
sense on any poset.
\end{remark}

The type of stratification Kashiwara and Schapira specify
\cite[Conjecture~3.17]{kashiwara-schapira2017} is not quite the same
as subanalytic subdivision in
Definition~\ref{d:auxiliary-hypotheses}.\ref{i:subanalytic}.  To be
precise, first recall two standard topological concepts.

\begin{defn}\label{d:locally-closed}
A subset of a topological space~$Q$ is \emph{locally closed} if it is
the intersection of an open subset and a closed subset.  A family of
subsets of~$Q$ is \emph{locally finite} if each compact subset of~$Q$
meets only finitely many members of the family.
\end{defn}

\begin{defn}[{\cite[Definition~3.15]{kashiwara-schapira2017}}]\label{d:stratification}
Fix $Q$ satisfying Hypothesis~\ref{h:full-closed-cone}.
\begin{enumerate}
\item%
A \emph{conic stratification} of a closed subset $S \subseteq Q$ is a
locally finite family of pairwise disjoint subanalytic subsets, called
\emph{strata}, which are locally closed in the conic topology and have
closures whose union is~$S$.

\item%
The stratification is \emph{subordinate} to a (derived) sheaf~$\FF$
on~$Q$ if $S$ equals the support of~$\FF$ and the restriction of (each
homology sheaf of)~$\FF$ to every stratum is locally constant of
finite~rank.
\end{enumerate}
\end{defn}

\begin{remark}\label{r:stratification}
A conic stratification is called a $\gamma$-stratification in
\cite[Definition~3.15]{kashiwara-schapira2017}, with $\gamma = Q_+$.
The only differences between conic stratification and subanalytic
partition of a subset~$S$ in
Definition~\ref{d:auxiliary-hypotheses}.\ref{i:subanalytic} are that
\begin{itemize}
\item%
conic stratifications are only required to be locally finite, not
necessarily finite;

\item%
conic strata are required to be locally closed in the conic topology
(that is, an intersection of an open upset in~$Q^\ord$ with a closed
downset in~$Q^\ord$); and

\item%
the union need not actually equal all of~$S$, because only the union
of the stratum closures is supposed to equal~$S$.
\end{itemize}
\end{remark}

\pagebreak[2]

\begin{prop}\label{p:stalks}
Fix a real partially ordered group~$Q$ with closed
positive~cone~$Q_+$.
\begin{enumerate}
\item\label{i:identity-continuous}%
The identity on~$Q$ yields continuous maps of topological spaces
$$%
  \iota: Q^\ord \to Q^\con
  \qquad\text{and}\qquad
  \jmath: Q^\ale \to Q^\con.
$$

\item\label{i:ord-pulled-back}%
Any sheaf~$\FF$ on~$Q^\ord$ pulled back from $Q^\con$ has natural maps
$$%
  \FF_q \to \FF_{q'}\text{ for }q \preceq q'\text{ in }Q
$$
on stalks that functorially define a $Q$-module $\bigoplus_{q\in Q}
\FF_q$.

\item\label{i:alexdrov=mod}%
Similarly, any sheaf~$\GG$ on~$Q^\ale$ has natural maps
$$%
  \GG_q \to \GG_{q'}\text{ for }q \preceq q'\text{ in }Q
$$
on stalks that functorially define a $Q$-module $\bigoplus_{q\in Q}
\GG_q$.  This functor from sheaves on~$Q^\ale$ to $Q$-modules is an
equivalence of categories.

\item\label{i:same-sheaf}%
If sheaves $\FF$ on~$Q^\ord$ and~$\GG$ on $Q^\ale$ are both pulled
back from the same sheaf $\EE$ on~$Q^\con$, then the $Q$-modules in
items~\ref{i:ord-pulled-back} and~\ref{i:alexdrov=mod} are the same.

\item\label{i:exact-pushforward}%
The pushforward functor $\jmath_*$ is exact, and $\jmath_* \jmath^{-1}
\EE \cong \EE$.
\end{enumerate}
\end{prop}
\begin{proof}
The maps in item~\ref{i:identity-continuous} are continuous by
definition: the inverse image of any open set is open because the
ordinary topology refines each of the target topologies.

For item~\ref{i:ord-pulled-back}, if $\FF = \iota^{-1}\EE$ is pulled
back to~$Q^\ord$ from a sheaf $\EE$ on~$Q^\con$, then $\FF$ has the
same stalks as~$\EE$ (as a sheaf pullback in any context does), so the
natural morphisms are induced by the restriction maps of~$\EE$ from
open neighborhoods of~$q$~to~those~of~$q'$.

The result in~\ref{i:alexdrov=mod} holds for arbitrary posets; for an
exposition in a context relevant to persistence, see
\cite[Theorem~4.2.10 and Remark~4.2.11]{curry-thesis} and
\cite{curry2019}.

For item~\ref{i:same-sheaf}, the stalks $\FF_q = \EE_q = \GG_q$ are
the same.

For item~\ref{i:exact-pushforward}, exactness is proved in passing in
the proof of \cite[Lemma~3.5]{berkouk-petit2019}, but it is also
elementary to check that a surjection $\GG \onto \GG'$ of sheaves
on~$Q^\ale$ yields a surjection of stalks for the pushforwards
to~$Q^\con$ because direct limits (filtered colimits) are exact.  That
$\jmath_* \jmath^{-1} \EE \cong \EE$ is because the natural morphism
is the identity on stalks.
\end{proof}

\subsection{Conic microsupport}\label{sub:microsupport}

\noindent
The \emph{microsupport} of a (derived) sheaf on an analytic
manifold~$X$ is a certain closed conic isotropic subset of the
cotangent bundle $T^*X$.  The notion of microsupport is a central
player in \cite{kashiwara-schapira1990}, to which the reader is
referred for background on the topic.  However, although the main
result in this section (Theorem~\ref{t:res}) is stated in terms of
microsupport, the next theorem allows the reader to ignore it
henceforth, as pointed out by Kashiwara and Schapira themselves
\cite[Remark~1.9]{kashiwara-schapira2018}, by immediately translating
to the more elementary context of sheaves in the conic topology in
Section~\ref{sub:topology}.

\begin{thm}[{\cite[Theorem~1.5 and Corollary~1.6]{kashiwara-schapira2018}}]\label{t:microsupport}
Fix $Q$ satisfying Hypothesis~\ref{h:full-closed-cone}.  The
pushforward $\iota_*$ of the map~$\iota$ from
Proposition~\ref{p:stalks}.\ref{i:identity-continuous} induces an
equivalence from the category of sheaves with microsupport contained
in the negative polar cone~$Q_+^\vee$ to the category of sheaves in
the conic topology.  The pullback~$\iota^{-1}$ is a quasi-inverse.
The same assertions hold for the bounded derived categories.
\end{thm}

\begin{remark}\label{r:push-and-pull}
The pushforward $\iota_*$ and the pullback $\iota^{-1}$ have concrete
geometric descriptions.  Since $\iota$ is the identity on~$Q$, the
pushforward of a sheaf~$\FF$ on $Q$ has~sections
$$%
  \Gamma(U, \iota_*\FF) = \Gamma(U,\FF)
$$
for any open upset~$U$, where ``open upset'' means the same things as
``upset that is open in the usual topology'' and ``subset that is open
in the conic topology''.  On the other hand, over any convex
ordinary-open set~$\OO$, the pullback to the ordinary topology of a
sheaf~$\EE$ in the conic topology has sections
$$%
  \Gamma(\OO,\iota^{-1}\EE) = \Gamma(\OO + Q_+, \EE),
$$
namely the sections of~$\EE$ over the upset generated by~$\OO$
\cite[(3.5.1)]{kashiwara-schapira1990}.
\end{remark}

\begin{remark}\label{r:microsupport}
What Theorem~\ref{t:microsupport} does in practice is allow a given
(derived) sheaf with microsupport contained in the negative polar
cone~$Q_+^\vee$ to be replaced with an isomorphic object that is
pulled back from the conic topology induced by the partial order.  The
reason for mentioning the notion of microsupport at all is to
emphasize that constructibility in the sense of
Definition~\ref{d:constructible-sheaf} requires the ordinary topology.
This may seem a fine distinction, but the conjectures of Kashiwara and
Schapira proved in Section~\ref{s:strat} entirely concern the
transition from the ordinary to the conic topology, so it is crucial
to be clear on this point.
\end{remark}

In view of Remark~\ref{r:microsupport}, discussion of constructibility
for sheaves on conic topologies requires the following.  The ad hoc
nature of this definition is justified by Theorem~\ref{t:res'}.

\begin{defn}\label{d:conic-sheaf}
Fix $Q$ satisfying Hypothesis~\ref{h:full-closed-cone}.  A
\emph{constructible conic sheaf} on~$Q$ is a sheaf in the conic
topology~$Q^\con$ whose pullback via $\iota^{-1}$ is subanalytically
constructible.
\end{defn}

\section{Resolutions of constructible sheaves}\label{s:res}

\begin{defn}\label{d:flasque}
Fix $Q$ satisfying Hypothesis~\ref{h:full-closed-cone}.
\begin{enumerate}
\item%
A \emph{subanalytic upset sheaf} on~$Q$ is the extension by zero of
the rank~$1$ constant sheaf on an open subanalytic upset in~$Q^\ord$.

\item%
A \emph{subanalytic downset sheaf} on~$Q$ is the pushforward of the
rank~$1$ locally constant sheaf on a closed subanalytic downset
in~$Q^\ord$.

\item%
A \emph{subanalytic upset resolution} of a complex~$\FF^\spot$ of
sheaves on~$Q^\ord$ is a homomorphism $\UU^\spot \to \FF^\spot$ of
complexes inducing an isomorphism on homology, with each $\UU^i$ being
a direct sum of subanalytic upset sheaves.

\item%
A \emph{subanalytic downset resolution} of a complex~$\FF^\spot$ of
sheaves on~$Q^\ord$ is a homomorphism $\FF^\spot \to \DD^\spot$ of
complexes inducing an isomorphism on homology, with each each $\DD^i$
being a direct sum of subanalytic downset sheaves.
\end{enumerate}
Either type of resolution is
\begin{itemize}
\item%
\emph{finite} if the total number of summands across all homological
degrees is finite;

\item%
\emph{PL} if $Q$ is polyhedral and the upsets or downsets are PL.
\end{itemize}
\end{defn}

\begin{prop}\label{p:upset}
Fix an upset $U$ in a real partially ordered group~$Q$ with closed
positive cone.  If $U^\circ$ is the interior of~$U$ in~$Q^\ord$, then
the sheaves on~$Q^\ale$ corresponding to~$\kk[U]$ and $\kk[U^\circ]$
push forward to the same sheaf on~$Q^\con$.
\end{prop}
\begin{proof}
The stalk at~$q$ of any sheaf on~$Q^\con$ is the direct limit over
points $p \in q - Q_+^\circ$ of the sections over $p + Q_+^\circ$.  In
the case of the pushforward of the sheaf on~$Q^\ale$ corresponding to
an upset module, these sections are~$\kk$ if $p$ lies interior to the
upset and~$0$ otherwise.  The result holds because the upsets $U$
and~$U^\circ$ have the same interior, namely~$U^\circ$.
\end{proof}

\begin{prop}\label{p:downset}
Fix a downset $D$ in a real partially ordered group~$Q$ with closed
positive cone.  If $\oD$ is the closure of $D$ in~$Q^\ord$, then the
sheaves on~$Q^\ale$ corresponding to~$\kk[D]$ and $\kk[\oD]$ push
forward to the same sheaf on~$Q^\con$.
\end{prop}
\begin{proof}
Calculating stalks as in the previous proof, in the case of the
pushforward of the sheaf on~$Q^\ale$ corresponding to a downset
module, the sections over $p + Q_+^\circ$ are~$\kk$ if $p$ lies
interior to the downset and~$0$ otherwise.  The result holds because
the downsets $D$ and~$\oD$ have the same interior.
\end{proof}

\begin{remark}\label{r:interior}
The fundamental difference between Alexandrov and conic topologies
reflected by\hspace{-.9pt} Propositions~\ref{p:upset}
and~\ref{p:downset} is explored in detail by\hspace{-.9pt} Berkouk
and~\mbox{Petit}~\cite{berkouk-petit2019}.
\end{remark}

Here is the main result.  It is little more than a restatement of the
relevant part of Theorem~\ref{t:syzygy-complexes} in the language of
sheaves.

\begin{thm}\label{t:res}
Fix $Q$ satisfying Hypothesis~\ref{h:full-closed-cone}.  If
$\FF^\spot$ is a complex of compactly supported subanalytically
constructible sheaves on~$Q^\ord$ with microsupport in the negative
polar cone~$Q_+^\vee$ then $\FF^\spot$ has a finite subanalytic upset
resolution and a finite subanalytic downset resolution.  If $Q$ is
polyhedral and $\FF^\spot$ is PL, then $\FF^\spot$ has PL such
resolutions.
\end{thm}
\begin{proof}
Using Theorem~\ref{t:microsupport}, assume that $\FF^\spot$ is pulled
back to~$Q^\ord$ from~$Q^\con$, say $\FF^\spot = \iota^{-1}\EE^\spot$.
Since $\FF^\spot$ has compact support, any subordinate subanalytic
triangulation (Definition~\ref{d:subanalytic-triangulation}) afforded
by Definition~\ref{d:constructible-sheaf} is necessarily finite
because it is locally finite.  The complex $F^\spot = \bigoplus_{q\in
Q} \FF^\spot_q$ of $Q$-modules that comes from
Proposition~\ref{p:stalks}.\ref{i:ord-pulled-back} is tamed by the
triangulation, which is a constant subdivision
(Definition~\ref{d:constant-subdivision}) because
\begin{itemize}
\item%
simplices are connected, so locally constant sheaves on them are
constant, and

\item%
$\Gamma(|\sigma_p|,\FF^i) \to \FF^i_p \to \FF^i_q \from
\Gamma(|\sigma_q|,\FF^i)$ is locally constant---and hence constant, as
simplices are connected---when $p \preceq q$ in~$Q$.  Here $\sigma_x$
is the simplex containing~$x$, the middle arrow is from
Proposition~\ref{p:stalks}.\ref{i:ord-pulled-back}, and the outer
arrows are the natural isomorphisms from Lemma~\ref{l:tri-constant}.
\end{itemize}
Hence the complex $F^\spot$ of $Q$-modules has resolutions of the
desired sort by \mbox{Theorem}~\ref{t:syzygy-complexes}.  Viewing any
of these resolutions as a complex of sheaves on~$Q^\ale$ via
Proposition~\ref{p:stalks}.\ref{i:alexdrov=mod}, push it forward from
the Alexandrov topology to the conic topology via the exact
functor~$\jmath_*$ in
Proposition~\ref{p:stalks}.\ref{i:exact-pushforward}.  The resulting
complex of sheaves on~$Q^\con$ is a resolution of a complex isomorphic
to~$\EE^\spot$ by Proposition~\ref{p:stalks}.\ref{i:same-sheaf}
and~\ref{p:stalks}.\ref{i:exact-pushforward}.  The upsets or downsets
in the summands of the resolution may as well be assumed open or
closed, respectively, by Propositions~\ref{p:upset}
or~\ref{p:downset}.  The proof is concluded by pulling back the
resolution from~$Q^\con$ to~$Q^\ord$ via the equivalence of
Theorem~\ref{t:microsupport}.
\end{proof}

\renewcommand\thethm{\arabic{section}.\arabic{thm}$'$}
\setcounter{separated}{\value{thm}}
\addtocounter{thm}{-1}
\begin{thm}\label{t:res'}
Fix $Q$ satisfying Hypothesis~\ref{h:full-closed-cone}.  If
$\FF^\spot$ is a complex of compactly supported sheaves in the conic
topology~$Q^\con$ then the following are equivalent.
\begin{enumerate}
\item%
$\FF^\spot$ is constructible (Definition~\ref{d:conic-sheaf}).

\item%
$\FF^\spot$ has a finite subanalytic upset resolution.

\item%
$\FF^\spot$ has a finite subanalytic downset resolution.
\end{enumerate}
The implications 2 $\implies$1 and 3 $\implies$1 do not
require
compact support for~$\FF^\spot$.  If $Q$ is polyhedral and $\FF^\spot$
is PL, then all of these claims hold
with ``PL'' in place of~``subanalytic''.
\end{thm}
\setcounter{thm}{\value{separated}}
\renewcommand\thethm{\arabic{section}.\arabic{thm}}
\begin{proof}
That 1 $\implies$ 2 and 1 $\implies$ 3 follows from
Theorems~\ref{t:res} and~\ref{t:microsupport}.  The opposite
directions are by the definition and foundational results surrounding
constructibility in Definition~\ref{d:constructible-sheaf} and
Remark~\ref{r:constructible-sheaf}.
\end{proof}

\begin{remark}\label{r:flexibility}
While the notion of a sheaf with microsupport contained in the
negative polar cone of~$Q_+$ is equivalent to the notion of a sheaf in
the conic topology, the notion of constructibility has until now only
been available on the microsupport side, where simplices from
arbitrary subanalytic triangulations achieve constancy of the sheaves
in question.  Theorem~\ref{t:res'} makes precise the assertion that
constructibility can be detected entirely with the more rigid conic
topology, without the flexibility of appealing to arbitrary
subanalytic triangulations.
\end{remark}

\begin{remark}\label{r:compact}
Theorem~\ref{t:res} assumes compact support to get finite instead of
locally finite subdivisions.  The application in
Section~\ref{s:strat} to constructible sheaves without any
assumption of compact support yields a locally finite subdivision by
reducing to the case of compact support.
\end{remark}

\begin{remark}\label{r:semialgebraic}
The final sentences of Theorems~\ref{t:res} and~\ref{t:res'} are true
with ``polyhedral'' and ``PL'' all replaced by ``semialgebraic'', with
the same proofs, as long as the definitions of these semialgebraic
concepts in the constructible sheaf setting are made appropriately.
The semialgebraic constructible sheaf versions are not treated here
because they are not relevant to the conjectures proved in
Section~\ref{s:strat}.
\end{remark}

\section{Stratifications from constructible sheaves}\label{s:strat}

\begin{cor}[{\cite[Conjecture~3.20]{kashiwara-schapira2019}}]\label{c:PL}
Fix $Q$ satisfying Hypothesis~\ref{h:full-closed-cone} with~$Q_+$
polyhedral.  If $\FF^\spot$ is a PL object in the derived category of
compactly supported constructible sheaves on~$Q^\ord$ with
microsupport contained in the negative polar cone~$Q_+^\vee$ then the
isomorphism class of~$\FF^\spot$ is represented by a complex that is a
finite direct sum of constant sheaves on bounded polyhedra that are
locally closed in the conic topology.
\end{cor}
\begin{proof}
The statement would directly be a special case of Theorem~\ref{t:res}
were it not for the boundedness hypothesis on the polyhedra, since
either a PL upset or PL downset resolution would satisfy the
conclusion.  That said, boundedness is easy to impose:
since~$\FF^\spot$ has compact support, and the resolution has
vanishing homology outside of the support of~$\FF^\spot$, each upset
or downset sheaf can be restricted to the support of~$\FF^\spot$ and
extended by~$0$.
\end{proof}

As in Theorem~\ref{t:res'}, Corollary~\ref{c:PL} can be restated using
constructible conic sheaves.

\renewcommand\thethm{\arabic{section}.\arabic{thm}$'$}
\addtocounter{thm}{-1}
\begin{cor}\label{c:PL'}
Fix $Q$ satisfying Hypothesis~\ref{h:full-closed-cone} with~$Q_+$
polyhedral.  $\FF^\spot$ is a PL object in the derived category of
compactly supported constructible conic sheaves if and only if the
isomorphism class of~$\FF^\spot$ is represented by a complex that is a
finite direct sum of constant sheaves on bounded polyhedra that are
locally closed in the conic topology.
\end{cor}
\renewcommand\thethm{\arabic{section}.\arabic{thm}}

\begin{cor}[{\cite[Conjecture\,3.17]{kashiwara-schapira2017}}]\label{c:strat}
Fix $Q$ satisfying Hypothesis~\ref{h:full-closed-cone}.  If a
compactly supported derived sheaf with microsupport in the negative
polar cone~$Q_+^\vee$ is subanalytically constructible, then its
support has a subordinate conic stratification.
\end{cor}
\begin{proof}
Part~(ii) in the proof of \cite[Theorem~3.17]{kashiwara-schapira2018}
reduces to the case where the support of the given derived sheaf is
compact.  The argument is presented in the case where $Q$ is
polyhedral and the derived sheaf is~PL, but the argument works
verbatim for $Q$ satisfying Hypothesis~\ref{h:full-closed-cone},
without any polyhedral or PL assumptions, because the requisite lemma,
namely \cite[Lemma~3.5]{kashiwara-schapira2018}---and indeed, all of
\cite[\S3.1]{kashiwara-schapira2018}---is stated and proved in
this non-polyhedral generality.  So henceforth assume the given
derived sheaf has compact support.

Remark~\ref{r:constructible-sheaf} allows the assumption that the
given derived sheaf is represented by a complex~$\FF^\spot$ of
constructible sheaves.  Theorem~\ref{t:res} produces a subanalytic
indicator resolution, which for concreteness may as well be an upset
resolution.  Each upset that appears as a summand in the resolution
partitions~$Q$ into the upset itself, which is open subanalytic, and
its complement, which is a closed subanalytic downset.  The common
refinement of the partitions induced by the finitely many open
subanalytic upsets in the resolution and their closed subanalytic
downset complements is a partition of~$Q$ into finitely many strata
such that
\begin{itemize}
\item%
each stratum is subanalytic and locally closed in the conic topology,
and

\item%
the restriction of~$\FF^\spot$ to each stratum has constant homology.
\end{itemize}
The strata with nonvanishing homology form the desired conic
stratification.
\end{proof}

As before, Corollary~\ref{c:strat} can be restated in terms of
constructible conic sheaves.

\renewcommand\thethm{\arabic{section}.\arabic{thm}$'$}
\addtocounter{thm}{-1}
\begin{cor}\label{c:strat'}
Fix $Q$ satisfying Hypothesis~\ref{h:full-closed-cone}.  The support
of any compactly supported constructible derived conic sheaf has a
subordinate conic stratification.
\end{cor}
\renewcommand\thethm{\arabic{section}.\arabic{thm}}

\begin{remark}\label{r:lambda}
The reference in \cite[Conjecture\,3.17]{kashiwara-schapira2017} to a
cone~$\lambda$ contained in the interior of the positive cone union
the origin appears to be unnecessary, since (in the notation there)
any $\gamma$-stratification is automatically a
$\lambda$-stratification by
\cite[Definition~3.15]{kashiwara-schapira2017} and the fact that
$\lambda \subseteq \gamma$.
\end{remark}

\bigskip
\setcounter{tocdepth}{-1}
\subsection*{Conflict of interest}\label{sub:conflict}
\setcounter{tocdepth}{2}

The author states that there is no conflict of interest.

\addtocontents{toc}{\protect\setcounter{tocdepth}{2}}


\end{document}